\documentclass{amsart}

\usepackage[utf8]{inputenc}
\usepackage{mathtools,amssymb}
\usepackage{dsfont}
\usepackage{tikz}
\usetikzlibrary{calc}
\usepackage{aliascnt}
\usepackage[colorlinks=true,linkcolor=red,citecolor=green,urlcolor=magenta]{hyperref}
\usepackage[nameinlink,noabbrev]{cleveref}
\usepackage{float}
\usepackage{orcidlink}

\theoremstyle{plain}
\newtheorem{theorem}{Theorem}[section]
\newaliascnt{lemma}{theorem}
\newtheorem{lemma}[lemma]{Lemma}
\aliascntresetthe{lemma}
\theoremstyle{plain}
\newaliascnt{remark}{theorem}
\newtheorem{remark}[remark]{Remark}
\aliascntresetthe{remark}
\crefname{lemma}{lemma}{lemmas}
\Crefname{lemma}{Lemma}{Lemmas}
\crefname{remark}{remark}{remarks}
\Crefname{remark}{Remark}{Remarks}
\numberwithin{equation}{section}

\renewcommand{\MR}[1]{\href{https://mathscinet.ams.org/mathscinet-getitem?mr=#1}{MR#1}}
\newcommand{\ARXIV}[1]{\href{https://arxiv.org/abs/#1}{arXiv:#1}}

\title[Lyapunov Exponents for Passive Scalars]{Lyapunov Exponents for Passive Scalars in White-in-Time Flows}
\author{Declan Stacy \orcidlink{0009-0006-7660-948X}}
\address{EPFL, Switzerland}
\email{declan.stacy@epfl.ch}
\date{September 2026}

\newcommand{\E}{\mathbb{E}}
\newcommand{\Prb}{\mathbb{P}}
\newcommand{\R}{\mathbb{R}}
\newcommand{\Id}{\mathds{1}}
\newcommand{\C}{\mathbb{C}}
\newcommand{\dd}{\mathrm{d}}
\DeclarePairedDelimiter{\ip}{\langle}{\rangle}
\newcommand{\T}{\mathbb{T}}
\newcommand{\Z}{\mathbb{Z}}
\newcommand{\Q}{\mathbb{Q}}

\begin{document}

\begin{abstract}
We study a passive scalar on the two-dimensional torus driven by an arbitrary finite set of white-in-time divergence-free Fourier modes. We prove that the exponential $L^2$ dissipation rate of the advection--diffusion equation is uniformly bounded as the diffusivity tends to zero as long as there exist two nonparallel forced modes. If not, then the dissipation rate can be arbitrarily fast depending on the initial data, but the top Lyapunov exponent remains uniformly bounded below. The proof uses a projection onto a finite set of modes which depends on the forcing set. Part of our argument follows the structure of the four-mode case investigated by Chemnitz and Chemnitz, but uses a strictly weaker condition than their matrix inequality between Frobenius and operator norms. Combined with an accessibility argument, this allows us to consider an arbitrary set of forcing modes and initial data.

\smallskip
\noindent\textbf{Keywords:} advection-diffusion, Lyapunov exponent, Batchelor scale.

\medskip
\noindent\textbf{2020 MSC: 35Q35, 37H15, 37L30, 76F25.}
\end{abstract}

\maketitle
\section{Introduction}
Consider the following equation:
\begin{equation}\label{eq:passive-scalar}
    \partial_t \theta_t = \kappa \Delta \theta_t - u_t \cdot \nabla \theta_t \,,
\end{equation}
where $\theta_t \in L^2_0(\T^2)$, the set of real-valued square-integrable mean-zero functions on $\T^2 = [-\pi,\pi]^2$, $\kappa > 0$ is called the diffusivity, and $u_t: \T^2 \to \R^2$ is some divergence-free (meaning $\partial_1 u_1 + \partial_2 u_2 = 0$) vector field which we will specify later. One can view $u_t$ as the velocity which pushes around some fluid in space and $\theta_t(x)$ as some scalar quantity, for example the deviation of the temperature from the mean temperature, which we observe at each point $x$ in space. In this temperature example, the $\kappa \Delta$ term represents the spreading around of heat.

The ``Batchelor Scale Conjecture'' \cite{Batchelor} concerns the distribution of the Fourier mass $|\hat \theta_t(k)|^2$ for $k \in \Z^2$, where
\begin{equation}\label{eq:fourier-transform-def}
    \hat \theta_t(k) \coloneqq \frac 1{(2\pi)^2}\int e^{-ik \cdot x}\theta_t(x)dx\,.
\end{equation}
Roughly, it is conjectured that $|\hat \theta_t(k)|^2 \sim |k|^{-2}$ for $|k| \leq \kappa^{-1/2}$ (assuming a stationary regime, which is possible when energy is pumped into the system via, for example, adding additive noise to \eqref{eq:passive-scalar}). Various forms of the conjecture have been proven for certain models. For example, \cite{cumulative} proves upper and lower bounds in expectation for $\sum_{|k| \leq R} |\hat \theta_t(k)|^2$ (of order $\ln R$) for $R \leq \kappa^{-1/2}$. \cite{annuli} proves analogous upper and lower bounds except they sum over annuli as opposed to balls, and summing their lower bounds recovers the lower bounds in \cite{cumulative}.

Another version of the Batchelor Scale Conjecture concerns the exponential decay rate of $\|\theta_t\|_{L^2} \coloneqq (2\pi)^{-1}\big(\int \theta_t(x)^2\dd x\big)^{1/2}$, known as the Lyapunov exponent:
\begin{equation}\label{eq:lambda-kappa-def}
    \lambda_\kappa \coloneqq \lambda_\kappa(\theta_0) \coloneqq \lim_{t \to \infty} \frac 1t \ln \|\theta_t\|_{L^2} \,,
\end{equation}
which may depend on $\theta_0$, the initial condition.
The relationship between $\lambda_\kappa$ and the distribution of the Fourier mass is the following formula, which is obtained simply by applying the chain rule to \eqref{eq:passive-scalar}:
\begin{equation}\label{eq:ln-theta}
    \ln \|\theta_t\|_{L^2} = \ln \|\theta_0\|_{L^2} - \kappa \int_0^t \frac{\|\nabla \theta_s\|_{L^2}^2}{\|\theta_s\|_{L^2}^2}ds\,, \quad  \frac{\|\nabla \theta_s\|_{L^2}^2}{\|\theta_s\|_{L^2}^2} = \frac{\sum_{k \in \Z^2} |k|^2 |\hat \theta_s(k)|^2}{\sum_{k \in \Z^2} |\hat \theta_s(k)|^2}\,.
\end{equation}
If we expect that on average we have most of the Fourier mass around modes with $|k| \sim \kappa^{-1/2}$ (the upper bound on $|k|$ from above), then \eqref{eq:ln-theta} suggests that $-\lambda_\kappa \asymp 1$.

Uniform upper bounds for $\lambda_\kappa$ (of the form $\lambda_\kappa \leq -C$) have been proven for various choices of $u_t$ in \cite{enhanced1, Gess} (in fact they prove even stronger results about the $L^2$ norm after a time of $O(\ln \kappa^{-1})$). See also \cite{FL26} for estimates on an averaged Lyapunov exponent in a different white-in-time transport-noise regime. However, uniform lower bounds on $\lambda_\kappa$ are much rarer in the current literature. In fact, even showing $\lambda_\kappa > -\infty$ is quite hard, and this is done in \cite{projective} through a precise analysis of the projective process $\theta_t/\|\theta_t\|_{L^2}$. Building upon the analysis of the spectral median introduced in \cite{projective-0}, the authors use a Wiener chaos expansion argument to show that Fourier mass on high-energy modes will quickly transfer to low ones. Their lower bound on $\lambda_\kappa$ goes to $-\infty$ as $\kappa \to 0$. On the other hand, \cite{fourmodes} proves a uniform lower bound on $\lambda_\kappa$, but only for an extremely specific choice of $u_t$.

In particular, suppose $u_t$ is white in time and smooth in space, for example
\begin{equation}\label{eq:white-in-time}
    d\theta_t(x) = \kappa \Delta \theta_t(x)dt + \sum_{k \in I}\sigma_k e^{ik \cdot x} k^\perp \cdot \nabla \theta_t(x) \circ dW^k_t \,,
\end{equation}
where $\circ$ denotes Stratonovich integration. Above, $I \subset \Z^2 \setminus \{0\}$ is a nonempty finite symmetric set ($k \in I \implies -k \in I$), $W^k$ are complex Brownian motions with $[W^k,W^\ell]_t = t\Id_{k = -\ell}$, the constants $\sigma_k \in \R \setminus \{0\}$ satisfy $\sigma_k = -\sigma_{-k}$, and we used the notation $k^\perp = (k_1,k_2)^\perp = (k_2,-k_1)$. Note that the choice $\sigma_k = -\sigma_{-k}$ ensures that $\theta_t$ is real-valued.

In \cite[Theorem 2.1]{fourmodes} the authors consider only the case $I = \{\pm (1,0), \pm (0,1)\}$ (as well as one other example on $\T^3$). By projecting $\theta_t$ onto the subspace spanned by Fourier modes in $I$ and applying It\^{o}'s formula to the log of the norm of the projection, they are able to show that $\lambda_\kappa \geq -C$ for some $C > 0$ independent of $\kappa \in (0,1]$, provided the initial condition $\theta_0$ is not orthogonal to that subspace. They rely crucially on a certain algebraic condition \cite[(24)]{fourmodes}, and they remark that finding other choices of $I$ satisfying the condition is quite difficult \cite[Remark 2.4]{fourmodes}. In our paper we modify their argument by projecting onto the subspace spanned by Fourier modes in a set $S$, where $S$ is not necessarily equal to $I$ and can depend on the initial data. We derive a relaxed version of the condition \cite[(24)]{fourmodes} in \Cref{lem:main-lem} and describe some sets $S$ which satisfy the relaxed condition. By constructing suitable sets $S$, we prove a result which applies to arbitrary sets $I$ and initial data $\theta_0$:
\begin{theorem}\label{thm:main}
    Let $I \subset \Z^2 \setminus \{0\}$ be an arbitrary nonempty finite symmetric set. Then there is some constant $C > 0$ (depending only on $I$ and $\sigma_k$) such that, for all $\kappa \in (0,1]$ and nonzero initial conditions $\theta_0 \in L_0^2(\T^2)$, the solution to \eqref{eq:white-in-time} satisfies:
    \begin{enumerate}
        \item If $I$ contains two nonparallel vectors, then
        \begin{equation}\label{eq:first-thm-assertion}
            \lambda_\kappa(\theta_0) \geq -C \quad \text{a.s.}
        \end{equation}
        \item If not, then $\operatorname{Span}_\Z(I) = q\Z$ for some $q \in \Z^2 \setminus \{0\}$. Let $\operatorname{supp} \hat \theta_0$ denote
        $\big\{n \in \Z^2 \setminus \{0\} \mid \hat \theta_0(n) \neq 0\big\}$. Then
        \begin{equation}\label{eq:second-thm-assertion}
        \begin{aligned}
            \lambda_\kappa(\theta_0) \geq -\inf\bigg(&\big\{C(q^\perp \cdot n)^2 \,\big|\, n \in \operatorname{supp} \hat \theta_0 \text{ and } q^\perp \cdot n \neq 0\big\} \\
       \cup\, &\big\{\kappa|n|^2 \,\big|\, n \in \operatorname{supp} \hat \theta_0 \text{ and } q^\perp \cdot n = 0\big\}\bigg) \quad \text{a.s.}
        \end{aligned}
        \end{equation}
    \end{enumerate}
\end{theorem}

\begin{remark}\label{rem:top-exp}
    Note that \Cref{thm:main} also implies the weaker statement that the top Lyapunov exponent, defined as $\lim_{t \to \infty} \frac 1t \ln \|\Phi(t,\omega)\|$ (where $\Phi(t,\omega)$ is as in \Cref{lem:met}), is bounded below by $-C$ (more precisely $-C|q|^2$ in the second case) for almost every $\omega$.
\end{remark}

\begin{figure}[H]
\centering
\begin{tikzpicture}[x=0.51cm,y=0.51cm,
    every node/.style={font=\small}]

\coordinate (O) at (0,0);
\coordinate (m) at (0,4);
\coordinate (n) at (10,8);

\coordinate (p) at ({40/29},{16/29});

\foreach \x in {-7,-6,...,15}{
  \foreach \y in {-4,-3,...,11}{
    \fill[black!28] (\x,\y) circle (1.0pt);
  }
}

\draw[black,line width=1.55pt]
  (-8,-3.2) -- (16,6.4);

\foreach \r in {-1,0,1,2}{
  \fill[black]
    ({5*\r},{2*\r}) circle (1.85pt);
}

\draw[blue!10,line width=1.75pt]
  (-8,0.8) -- (16,10.4);

\draw[blue!30,line width=1.85pt]
  (-7.5,1.0) -- (7.5,7.0);

\draw[blue!55,line width=1.95pt]
  (-5,2.0) -- (5,6.0);

\draw[blue!85!black,line width=2.05pt]
  (-2.5,3.0) -- (2.5,5.0);

\foreach \r in {-1,0,1,2}{
  \fill[blue!70!black]
    ({5*\r},{4+2*\r}) circle (1.85pt);
}

\fill[black] (O) circle (2.45pt);
\fill[black] (m) circle (2.75pt);
\fill[black] (n) circle (2.75pt);
\fill[black] (p) circle (1.75pt);

\draw[black,densely dotted,line width=1.15pt]
  (m) -- (O);

\draw[black,densely dotted,line width=1.15pt]
  (m) -- (p);

\coordinate (ra) at ($(p)+({0.42*5/sqrt(29)},{0.42*2/sqrt(29)})$);
\coordinate (rb) at ($(ra)+({0.42*(-2)/sqrt(29)},{0.42*5/sqrt(29)})$);
\coordinate (rc) at ($(p)+({0.42*(-2)/sqrt(29)},{0.42*5/sqrt(29)})$);

\draw[black,line width=0.8pt]
  (ra) -- (rb) -- (rc);

\node[below left=5pt] at (O) {$0$};
\node[above left=6pt] at (m) {$m$};
\node[above left=5pt] at (n) {$n$};

\node[black,above=7pt]
  at (13.2,5.28) {$q\mathbb Z$};

\node[blue!70!black,above=7pt]
  at (12.1,8.84) {$n+q\mathbb Z$};

\node[black,anchor=west]
  at (0.88,2.35)
  {$\sim |q^\perp\!\cdot n|$};

\end{tikzpicture}

\caption{The drawing above illustrates the first case of \eqref{eq:second-thm-assertion}. In particular, if $\operatorname{supp} \hat \theta_0$ is contained on the lines $\pm n + q\Z$ where $q^\perp \cdot n \neq 0$, then \eqref{eq:dv-eq-1} shows that the Fourier mass will remain supported on those lines for all $t$. The heuristic interpretation of \eqref{eq:second-thm-assertion} is that eventually the Fourier mass will distribute along the lines $\pm n + q\Z$, as depicted by the blue shading. In particular, there will eventually be mass on mode $m$, where $m$ minimizes $|m|^2$, enough so that the dissipation rate $-\lambda_\kappa$ is comparable to $|m|^2 \sim (q^\perp \cdot m)^2 = (q^\perp \cdot n)^2$. On the other hand, when $\operatorname{supp} \hat \theta_0 = \{\pm n\}$ where $q^\perp \cdot n = 0$, meaning $n + q\Z \subset q\R$, the mass does not distribute and thus we get the dissipation rate $\kappa |n|^2$ (see \Cref{rem:failure-to-be-uniform}).}
\label{fig:rank-one-coset}
\end{figure}

\begin{remark}\label{rem:failure-to-be-uniform}
    When $\operatorname{Span}_\Z(I) = q\Z$, $\lambda_\kappa$ can be arbitrarily negative depending on the initial data. For example, when $\theta_0(x) = 2\cos(n \cdot x)$ where $q^\perp \cdot n = 0$, from \eqref{eq:dv-eq-1} we have the exact formula $v_n(t) = e^{-\kappa |n|^2 t}$ (note that $c_n = 0$ and thus the It\^{o} term vanishes as well), and so $\lambda_\kappa = -\kappa|n|^2$. Thus, it is impossible to strengthen \Cref{thm:main} to say that $\lambda_\kappa \geq -C$ regardless of $I$.
\end{remark}

\begin{remark}
   If we impose the additional condition that $\operatorname{Span}_\Z(I) = \Z^2$, there is an even simpler proof of $\lambda_\kappa \geq -C$, which we discuss in \Cref{rem:alternative-proof}. In particular, there is a weighted tweak to the projection considered in \cite{fourmodes} which allows \cite[(24)]{fourmodes} to hold regardless of $I$. For arbitrary $I$ this alternative proof also works to show the uniform lower bound on the top Lyapunov exponent discussed in \Cref{rem:top-exp}. However, this construction does not imply \Cref{thm:main} for arbitrary $I$ --- for example it would not work for $I = \{\pm(2,0),\pm(0,2)\}$ --- because \Cref{thm:main} concerns arbitrary initial data, and this alternative proof requires $|\hat \theta_0(n)| \neq 0$ for some $n \in \operatorname{Span}_\Z(I)$. 
\end{remark}

\begin{remark}
    If $0 \neq \theta_0 \in L^2_0(\T^2)$ is fixed, then \Cref{thm:main} implies that there is some $C(\theta_0) \in (0,\infty)$ such that, for any sequence $\kappa_n \downarrow 0$,
    \[\liminf_{n \to \infty} \lambda_{\kappa_n}(\theta_0) \geq -C(\theta_0) \quad \text{a.s.}\,,\] which is not immediately obvious from \eqref{eq:white-in-time}.
\end{remark}

In \Cref{sec:key-estimate} we prove a quantitative lower bound on $\lambda_\kappa(\theta_0)$ under the assumption $P_S\theta_0 \neq 0$, where $P_S$ is a suitable projection depending on $I$ and $\operatorname{supp} \hat \theta_0$ (\Cref{lem:old-thm}). In \Cref{sec:proof-section} we finish the proof of \Cref{thm:main} with an accessibility argument. In \Cref{sec:appendix} we recall a martingale limit theorem and a version of the multiplicative ergodic theorem which are used in the proof.

\section{Key Estimate}\label{sec:key-estimate}
Consider sets $S$ of the form
\begin{equation}\label{eq:S-def}
    S \coloneqq \Bigg\{\pm m+\sum_{k \in I} a_k k \,\Bigg|\, a_k \in \{0,1\}\Bigg\} \setminus \{0\}
\end{equation}
for some $m \in \Z^2$. The goal of this section is to prove the following quantitative lower bound on $\lambda_\kappa$:
\begin{lemma}\label{lem:old-thm}
      Let $I \subset \Z^2 \setminus \{0\}$ be an arbitrary nonempty finite symmetric set and $S$ be as in \eqref{eq:S-def} for some $m \in \Z^2$. Then for all $\kappa \in (0,1]$ and all initial conditions $\theta_0 \in L_0^2(\T^2)$ such that $S \cap \operatorname{supp} \hat \theta_0$ is nonempty, the solution to \eqref{eq:white-in-time} satisfies
    \begin{equation}\label{eq:lambda-kappa}
        \lambda_\kappa(\theta_0) \geq -\sup_{n \in S} \Bigg(\kappa |n|^2 + \frac 12 \sum_{k \in I} \sigma_k^2(k^\perp \cdot n)^2\Bigg) \quad \text{a.s.}
    \end{equation}
\end{lemma}

The only property of $S$ we will use (besides being finite and symmetric) is the following:
\begin{lemma}\label{lem:S-prop}
For all $k = (k_1,k_2) \in I$ and $n \in S$,
    \[n - k \notin S \text{ and } k^\perp \cdot n \neq 0 \implies n + k \in S\,,\]
    where $k^\perp$ denotes $(k_2,-k_1)$.
\end{lemma}
\begin{proof}
    Write $n = d_mm + \sum_{j \in I} a_j j$ where $a_j \in \{0,1\}$ and $d_m \in \{\pm 1\}$. Suppose $n - k \notin S$ and $k^\perp \cdot n \neq 0$. If $a_k = 0$, then $n + k = d_m m+\sum_{j \in I} (a_j + \Id_{j = k})j \in S$ unless $n + k = 0$, but this would contradict $k^\perp \cdot n \neq 0$. If instead $a_k = 1$, then $n - k = d_m m + \sum_{j \in I} (a_j - \Id_{j = k})j \notin S$ implies that $n -k = 0$, which again contradicts $k^\perp \cdot n \neq 0$.
\end{proof}

Recall \eqref{eq:fourier-transform-def} and set $v_n(t) \coloneqq \hat \theta_t(n)$. Then \eqref{eq:white-in-time} can be rewritten as
\begin{equation}\label{eq:dv-eq-1}
    \begin{aligned}
       dv_n &= -\kappa|n|^2v_n dt + i\sum_{k \in I} \sigma_k(k^\perp \cdot n)v_{n-k} \circ dW^k \\
       &= \Bigg[-\kappa|n|^2v_n - \frac 12 c_nv_n\Bigg]dt + i\sum_{k \in I} \sigma_k(k^\perp \cdot n)v_{n-k}  dW^k\,,
    \end{aligned}
\end{equation}
where $c_n \coloneqq \sum_{k \in I} \sigma_k^2(k^\perp \cdot n)^2$ and the second line is written in It\^{o} form.
Noticing that the coefficients in front of $dW^k$ are linear in $v$, we rewrite \eqref{eq:dv-eq-1} as
\begin{equation}\label{eq:dv-eq-2}
    dv_n = \Bigg[-\kappa|n|^2v_n - \frac 12 c_nv_n\Bigg] dt + (A(v) dW)_n\,,
\end{equation}
where formally $A(v): H_I \to H$, $W = (W^k)_{k \in I} \in H_I$, $H \coloneqq H_{\Z^2 \setminus \{0\}}$, and for any nonempty symmetric $J \subset \Z^2 \setminus \{0\}$ we define the real Hilbert space
\begin{equation}\label{eq:h-i-def}
    H_J \coloneqq \bigg\{w \in \ell_2(J;\C) \,\bigg|\, w_k = \overline{w_{-k}} \text{ for all } k \in J \bigg\}
\end{equation}
with inner product $\ip{w,v} \coloneqq \sum_{k \in J} \overline{w_k}v_k$. Defining $(A(v)x)_n \coloneqq \sum_{k \in I} A(v)_{n,k}x_k$, we can express $A(v)$ via $A(v)_{n,k} = i\sigma_k(k^\perp \cdot n)v_{n-k}$.

Denote the projection from $H$ to $H_S$ by $P_S$, so $P_Sv \coloneqq (v_n)_{n \in S} \in H_S$, and let $\tilde A(v) \coloneqq P_SA(v)$. By applying It\^{o}'s formula to \eqref{eq:dv-eq-2} we obtain
\begin{equation}\label{eq:d-ln-pi}
    \begin{aligned}
    d\ln \|P_S v\|
    &= \left[-\kappa\frac{\sum_{n \in S} |n|^2|v_n|^2}{\sum_{n \in S}|v_n|^2}-\frac{1}{2}\frac{\sum_{n \in S} c_n|v_n|^2}{\sum_{n \in S}|v_n|^2}\right]dt \\
    &+\frac{\|P_S v\|^2\|\tilde A(v)\|_{\mathrm{Fr}}^2
    -2\left\|\tilde A(v)^*P_S v\right\|^2}{2\|P_S v\|^4}
   dt
    +\frac{\ip{P_S v,\tilde A(v)\,dW}}{\|P_S v\|^2} \,,
\end{aligned}
\end{equation}
where $\|\cdot\|_{\mathrm{Fr}}$ denotes the Frobenius norm $\|A\|_{\mathrm{Fr}}^2 = \sum_{n \in S,k \in I} |A_{n,k}|^2$.

The key ingredient of the proof of \Cref{lem:old-thm} is the following:
\begin{lemma}\label{lem:main-lem}
    For all $v \in H$ with $P_S v \neq 0$, the second term in \eqref{eq:d-ln-pi} is nonnegative:
    \begin{equation}\label{eq:nonneg-ito-correction}
        \|P_S v\|^2\|\tilde A(v)\|_{\mathrm{Fr}}^2
     \geq 2\left\|\tilde A(v)^*P_S v\right\|^2\,.
    \end{equation}
\end{lemma}

\begin{proof}
First decompose $\tilde A(v) = A_1(v) + A_2(v)$, defined for $(n,k) \in S \times I$ by
\begin{align*}
    (A_1(v))_{n,k} &\coloneqq i\sigma_k(k^\perp \cdot n)v_{n-k}\Id_{n - k \notin S} \\
    (A_2(v))_{n,k} &\coloneqq i\sigma_k(k^\perp \cdot n)v_{n-k}\Id_{n - k \in S}\,.
\end{align*}
Then $A_1(v)_{n,k}A_2(v)_{n,k} = 0$ for all $(n,k)$, so
\begin{equation}\label{eq:frob-norm-bd}
    \|\tilde A(v)\|_{\mathrm{Fr}}^2 = \|A_1(v)\|_{\mathrm{Fr}}^2 + \|A_2(v)\|_{\mathrm{Fr}}^2 \geq \|A_1(v)\|_{\mathrm{Fr}}^2\,.
\end{equation}
We also have $A_2(v)^*P_Sv = 0$. Indeed, first replacing $n$ with $-n$ and then $n$ with $n-k$, using $v_n = \overline v_{-n}$ throughout, gives
\begin{align*}
    (A_2(v)^* P_S v)_k
    &= -i\sigma_k \sum_{n \in S} (k^\perp \cdot n)\overline v_{n-k}v_n\Id_{n-k \in S} = i\sigma_k \sum_{n\in S} (k^\perp \cdot n) v_{n+k}\overline v_n\Id_{n+k \in S} \\
    &= i\sigma_k \sum_{n-k \in S} (k^\perp \cdot n)v_n\overline v_{n-k} \Id_{n \in S} = i\sigma_k \sum_{n \in S} (k^\perp \cdot n)v_n\overline v_{n-k} \Id_{n-k \in S}\,,
\end{align*}
and notice that the second and fifth expressions are negatives. Thus,
\begin{equation}\label{eq:adjoint-bd-1}
    \begin{aligned}
        2\left\|\tilde A(v)^*P_S v\right\|^2 &= 2\|A_1(v)^*P_Sv\|^2 \\
        &= \sum_{k \in I}\Big(|\ip{A_1(v)_k, P_S v}|^2 + |\ip{A_1(v)_{-k}, P_S v}|^2\Big)\,,
    \end{aligned}
\end{equation}
where $A_1(v)_k$ denotes the $k$th column of $A_1(v)$ and for the last equality recall that $I$ is symmetric. Next, note that by \Cref{lem:S-prop} we have
\begin{equation*}
    (A_1(v))_{n,k} \neq 0 \implies  n-k \notin S \text{ and } k^\perp \cdot n \neq 0 \implies n+k \in S \implies (A_1(v))_{n,-k} = 0\,,
\end{equation*}
so we conclude that $\ip{A_1(v)_k, A_1(v)_{-k}} = 0$. Replacing $n$ by $-n$ and using the symmetry of $S$ and $v_n = \overline v_{-n}$, we also obtain 
\begin{align*}
    \|A_1(v)_k\|^2 &= \sigma_k^2\sum_{n \in S}(k^\perp \cdot n)^2|v_{n-k}|^2\Id_{n - k \notin S} \\
    &= \sigma_k^2\sum_{n \in S}(k^\perp \cdot n)^2|v_{-n-k}|^2\Id_{-n - k \notin S} = \|A_1(v)_{-k}\|^2\,.
\end{align*}
For $k$ such that $\|A_1(v)_k\| \neq 0$, we conclude that
\[
A_1(v)_k/\|A_1(v)_k\| \quad\text{and}\quad A_1(v)_{-k}/\|A_1(v)_k\|
\]
are orthonormal, so
\begin{equation}\label{eq:adjoint-bd-2}
    \begin{aligned}
         \sum_{k \in I}\Big(|\ip{A_1(v)_k, P_S v}|^2 + |\ip{A_1(v)_{-k}, P_S v}|^2\Big) &\leq \sum_{k \in I} \|A_1(v)_k\|^2\|P_Sv\|^2 \\
         &= \|A_1(v)\|_{\mathrm{Fr}}^2\|P_Sv\|^2\,.
    \end{aligned}
\end{equation}
Combining \eqref{eq:frob-norm-bd}--\eqref{eq:adjoint-bd-2} yields \eqref{eq:nonneg-ito-correction}.
\end{proof}

\begin{remark}\label{rem:frobenius-failure}
The stronger Frobenius/operator-norm condition used in \cite[(24)]{fourmodes} need not hold for our choice of $S$, even for their four-mode forcing $I = \{\pm(1,0), \pm (0,1)\}$. For $m = 0$, the set in \eqref{eq:S-def} is
\[
S = \{\pm(1,0),\pm(0,1),\pm(1,1),\pm(1,-1)\},
\]
and for ease of computation take $\sigma_k = \Id_{k_1 + k_2 > 0} - \Id_{k_1 + k_2 < 0}$.

Consider $v\in H$ given by
\[
v_{(0,1)}=v_{(0,-1)}=1,
\qquad
v_{(2,-1)}=i,
\qquad
v_{(-2,1)}=-i,
\]
with all other coefficients equal to zero. Recall that $\tilde A(v)_{n,k} = i\sigma_k(k^\perp\cdot n)v_{n-k}$, which is $0$ when $k \in \{\pm(0,1)\}$, $n \in S$ because none of the sums $(0,1)\pm(0,1)$ or $(2,-1)\pm(0,1)$ lie in $S$. $\tilde A(v)_{n,k}$ is also $0$ when $k,n \in I$ since none of the sums $(0,1)\pm(1,0)$ or $(2,-1)\pm(1,0)$ lie in $I$. Thus, we may restrict ourselves to $k \in \{\pm(1,0)\}$ and $n \in \{\pm(1,1),\pm(1,-1)\}$ to obtain the $4\times 2$ matrix
\[
\begin{aligned}
&
\begin{pmatrix}
i\sigma_{(1,0)}\big((1,0)^\perp\cdot(1,1)\big)v_{(0,1)}
&
i\sigma_{(-1,0)}\big((-1,0)^\perp\cdot(1,1)\big)v_{(2,1)}
\\[2mm]
i\sigma_{(1,0)}\big((1,0)^\perp\cdot(1,-1)\big)v_{(0,-1)}
&
i\sigma_{(-1,0)}\big((-1,0)^\perp\cdot(1,-1)\big)v_{(2,-1)}
\\[2mm]
i\sigma_{(1,0)}\big((1,0)^\perp\cdot(-1,1)\big)v_{(-2,1)}
&
i\sigma_{(-1,0)}\big((-1,0)^\perp\cdot(-1,1)\big)v_{(0,1)}
\\[2mm]
i\sigma_{(1,0)}\big((1,0)^\perp\cdot(-1,-1)\big)v_{(-2,-1)}
&
i\sigma_{(-1,0)}\big((-1,0)^\perp\cdot(-1,-1)\big)v_{(0,-1)}
\end{pmatrix}
\\[3mm]
&\hspace{3cm}
=
\begin{pmatrix}
-i & 0\\
 i & -1\\
-1 & -i\\
 0 & i
\end{pmatrix}\,.
\end{aligned}
\]
Consequently, $\|\tilde A(v)\|_{\mathrm{Fr}}^2 = 6$. On the other hand, let $x\in H_I$ be defined by
\[
x_{(1,0)}=\frac{1+i}{2},
\qquad
x_{(-1,0)}=\frac{1-i}{2},
\qquad
x_{\pm (0,1)}=0.
\]
Then $\|x\|=1$, while
\[
\tilde A(v)x
=
\frac12
\begin{pmatrix}
1-i\\
-2+2i\\
-2-2i\\
1+i
\end{pmatrix}
\]
on the four rows displayed above. Hence $\|\tilde A(v)x\|^2=5$ and therefore
\[2\|\tilde A(v)\|_{\mathrm{op}}^2\ge 10 > 6 = \|\tilde A(v)\|_{\mathrm{Fr}}^2\,,\]
which violates \cite[(24)]{fourmodes}. Notice, however, that $(P_Sv)_n = 0$ for $n \notin \{\pm (0,1)\}$, and we noted above that $\tilde A(v)_{n,k} = 0$ for $n \in \{\pm (0,1)\}$, hence $\tilde A(v)^*P_Sv=0$. Thus, our weaker inequality \eqref{eq:nonneg-ito-correction} holds trivially for this $v$.
This demonstrates the key distinction between \cite[(24)]{fourmodes}, which controls $\|\tilde A(v)^*w\|$ uniformly over all $w\in H_S$, and \eqref{eq:nonneg-ito-correction}, which only looks at the particular vector $w=P_Sv$.
\end{remark}

\begin{remark}\label{rem:alternative-proof}
    The stronger Frobenius/operator-norm condition used in \cite[(24)]{fourmodes} does hold for a different choice of weighted projection, which simplifies our proof in the case where $\operatorname{Span}_\Z(I) = \Z^2$. Indeed, instead of $P_S: H \to H_S$, consider the map $P: H \to H_I$ given by $Pv = (i\sigma_n v_n)_{n \in I} \in H_I$. With $\tilde A(v) \coloneqq PA(v)$, we have $(\tilde A(v))_{n,k} = i\sigma_n [i \sigma_k (k^\perp \cdot n)v_{n-k}] = -\sigma_n\sigma_k(k^\perp \cdot n)v_{n-k}$. Since $k^\perp \cdot n = -n^\perp \cdot k$ and $\overline{v_{n-k}} = v_{k-n}$, we have in fact $\overline{(\tilde A(v))_{n,k}} = -(\tilde A(v))_{k,n}$. Thus, $\tilde A(v)$ is skew-symmetric as an operator on $H_I \cong \R^{|I|}$ and so its singular values come in pairs $r_1,r_1,r_2,r_2,\dots,r_{|I|/2},r_{|I|/2}$. Since $\|\tilde A(v)\|_{\mathrm{op}}^2 = \max_i |r_i|^2$ and $\|\tilde A(v)\|_{\mathrm{Fr}}^2 = 2\sum_{i=1}^{|I|/2} r_i^2$, we indeed have $\|\tilde A(v)\|_{\mathrm{Fr}}^2 \geq 2\|\tilde A(v)\|_{\mathrm{op}}^2$. Repeating the argument below thus gives a lower bound on $\lambda_\kappa$ whenever $Pv(0) \neq 0$. Since \Cref{lem:acc} and \Cref{lem:rm-characterization} below show that almost surely there is some $t$ such that $Pv(t) \neq 0$, this proves \eqref{eq:first-thm-assertion} for all nonzero initial data. However, if $\operatorname{Span}_\Z(I) \neq \Z^2$ then we may have $Pv(t) = 0$ for all $t \geq 0$ (for example when $I = \{\pm(2,0),\pm(0,2)\}$ and $\theta_0(x) = \cos(x_1)$).
\end{remark}

Finally, we prove \Cref{lem:old-thm} by following the proof of \cite[Theorem 2.1]{fourmodes}, using \Cref{lem:main-lem} in place of \cite[(24)]{fourmodes}:

\begin{proof}
Let $\kappa \in (0,1]$ and $\theta_0 \in L^2_0(\T^2)$ be such that $S \cap \operatorname{supp} \hat \theta_0$ is nonempty, which is equivalent to $P_Sv(0) \neq 0$. We may assume $\|\theta_0\|_{L^2} = 1$ without loss of generality. Letting $\tau = \inf\{t \geq 0 \mid P_Sv(t) = 0\}$, we have by \eqref{eq:d-ln-pi} for all $t < \tau$ that
\begin{align}
        \ln \|P_Sv(t)\| &- \ln\|P_Sv(0)\| \notag\\
        &= - \int_0^t \frac{\sum_{n \in S}(\kappa |n|^2 + c_n/2)|v_n(s)|^2}{\sum_{n \in S}|v_n(s)|^2}\dd s + A_t +M_t\,,\label{eq:long-ln-formula}
\end{align}
where $A_t \geq 0$ by \Cref{lem:main-lem} and $M_t$ is a continuous local martingale on $[0,\tau)$ with
\begin{equation}\label{eq:quad-var}
    \langle M\rangle_t = \int_0^t \frac{\|\tilde A(v(s))^*P_Sv(s)\|^2}{\|P_Sv(s)\|^4}\dd s \lesssim \int_0^t \frac{1}{\|P_Sv(s)\|^2}\dd s
\end{equation}
(the bound on $\tilde A(v)$ follows from $|v_n(t)| \leq \|\theta_0\|_{L^2} = 1$, which is true by \eqref{eq:ln-theta}). By $\|P_Sv(t)\| \leq \|\theta_t\|_{L^2} \leq \|\theta_0\|_{L^2} = 1$, \eqref{eq:ln-theta}, $A_t \geq 0$, and \eqref{eq:long-ln-formula} we have (for all $t < \tau$)
\begin{equation}\label{eq:almost-done}
      0 \geq  \ln \|\theta_t\|_{L^2} \geq \ln \|P_Sv(t)\| 
       \geq \ln \|P_Sv(0)\|- C_\kappa t + M_t \,,
\end{equation}
where $C_\kappa \coloneqq \sup_{n \in S} \big(\kappa |n|^2 + \frac 12 \sum_{k \in I} \sigma_k^2(k^\perp \cdot n)^2\big)$. Next we show that $\tau = \infty$, so that \eqref{eq:almost-done} holds for all $t$. Indeed, let $\tau_N \coloneqq \inf\{ t\geq 0 \mid \|P_Sv(t)\| \leq e^{-N}\}$. Then by \eqref{eq:quad-var} we have $\E[M_{t \wedge \tau_N}] = 0$, so by \eqref{eq:almost-done}
\begin{align*}
    \Prb(\tau_N \leq t) &= \Prb\Big(-\ln \|P_Sv(t \wedge \tau_N)\| \geq N\Big) \\
    &\leq \frac{\E\big[-\ln \|P_Sv(t \wedge \tau_N)\|\big]}{N} \leq \frac{C_\kappa t - \ln \|P_Sv(0)\|}{N}\,,
\end{align*}
where we used $-\ln \|P_Sv(t \wedge \tau_N)\| \geq 0$. Thus, $\Prb(\tau_N \leq t) \to 0$ as $N \to \infty$, so $\sup_N \tau_N = \infty$ almost surely. Since $\tau \geq \sup_N \tau_N$, we indeed have $\tau = \infty$.

Finally, by \eqref{eq:almost-done} and \Cref{lem:martingale}, almost surely
\[\limsup_{t \to \infty} \frac{1}{t} \ln \|\theta_t\|_{L^2} \geq -C_\kappa\,,\]
which by \Cref{lem:met} proves \eqref{eq:lambda-kappa}.
\end{proof}

\section{Proof of \texorpdfstring{\Cref{thm:main}}{Theorem 1.1}}\label{sec:proof-section}
We wish to apply \Cref{lem:old-thm} to prove \Cref{thm:main}, so we must choose $m$ in \eqref{eq:S-def} such that $S \cap \operatorname{supp} \hat \theta_0$ is nonempty. Since $m \in S$, the naive choice would be to choose an $m \in \operatorname{supp} \hat \theta_0$ which minimizes
\[\sup_{n \in S} \Bigg(\kappa |n|^2 + \frac 12 \sum_{k \in I} \sigma_k^2(k^\perp \cdot n)^2\Bigg)\,.\]
However, this does not yield a lower bound which is uniform in $\theta_0$ because $\kappa|m|^2$ may be arbitrarily large. This idea can be improved by an accessibility lemma:
\begin{lemma}\label{lem:acc}
    Define an undirected graph on $\Z^2 \setminus \{0\}$ with edge set
    \[E \coloneqq \big\{(n,n+k) \mid n \in \Z^2 \setminus \{0\}\,, k \in I\,, k^\perp \cdot n \neq 0\big\}\,,\]
    and for $n \in \Z^2 \setminus \{0\}$ let $R(n)$ denote the connected component containing $n$. Let $\theta_0 \in L^2_0(\T^2) \setminus\{0\}$. Then for all $n \in \operatorname{supp} \hat \theta_0$ and $m \in R(n)$, the solution to \eqref{eq:white-in-time} satisfies
    \[\Prb\big(\tau_m = 0\big) = 1\,, \quad \text{where} \quad \tau_m \coloneqq \inf\big\{t \geq 0 \,\big|\, m \in \operatorname{supp} \hat \theta_t\big\}\,.\]
\end{lemma}
\begin{proof}
     First we show the claim for $m = n+k$ where $k \in I$ and $k^\perp \cdot n = k^\perp \cdot m \neq 0$. In other words, $(n,m) \in E$. By \eqref{eq:dv-eq-1} we have
    \[v_m(t) = e^{-c_{\kappa,m}t}v_m(0) + i\sum_{\ell \in I}\sigma_\ell(\ell^\perp \cdot m)\int_0^t e^{-c_{\kappa,m}(t-s)} v_{m-\ell}(s)dW^\ell_s\,,\]
    where $c_{\kappa,m} \coloneqq \kappa|m|^2 + c_m/2 > 0$ and recall the notation $v_m(t) = \hat \theta_t(m)$. If $v_m(0) \neq 0$ then $\tau_m = 0$. Otherwise $N_t \coloneqq e^{c_{\kappa,m}t}v_m(t)$ is a continuous local martingale with
    \[[N, \overline N]_t = \sum_{\ell \in I} \sigma_\ell^2 (\ell^\perp \cdot m)^2 \int_0^t e^{2c_{\kappa,m}s}|v_{m-\ell}(s)|^2\dd s\,.\]
    Using our assumption $n \in \operatorname{supp} \hat \theta_0$, which means $v_n (0)= v_{m-k}(0) \neq 0$, almost surely $[N, \overline N]_t > 0$ for all $t > 0$, and thus $\Prb(N_s = 0 \text{ for all }s \leq t) = 0$, which proves the claim.

     For a general $m \in R(n)$ there exist edges $(n_0,n_1),\dots,(n_{N-1},n_N) \in E$ such that $n_0 = n$ and $n_N = m$.
    For $\epsilon > 0$ and $i \in \{0,\dots,N\}$, assume
    \begin{equation}\label{eq:prb-equals-1}
         \Prb(A_i) \coloneqq \Prb\big(\exists t \in [0,\epsilon) \cap \Q \text{ such that } n_i \in \operatorname{supp} \hat \theta_{t}\big) = 1 \,,
     \end{equation}
    which is true for $i = 0$ since $n_0 = n \in \operatorname{supp} \hat \theta_0$. By the $(n,m) \in E$ case proved above, the Markov property, and continuity of $t \mapsto \hat \theta_t(n_i)$, we have $\Prb(A_{i+1}^c \cap \{n_i \in \operatorname{supp} \hat \theta_{t}\}) = 0$ for all $t \in [0,\epsilon) \cap \Q$. In other words, $n_i \in \operatorname{supp} \hat \theta_{t}$ implies $A_{i+1}$ almost surely. Combined with \eqref{eq:prb-equals-1}, this shows $\Prb(A_{i+1}) = 1$, and thus $\Prb(A_N) = 1$. Taking $\epsilon \downarrow 0$ gives $\Prb(\tau_m = 0) = 1$.
\end{proof}
Next we characterize $R(n)$:
\begin{lemma}\label{lem:rm-characterization}
    Let $n \in \Z^2 \setminus \{0\}$.
    If $\operatorname{Span}_\Z(I) = q\Z$ for some $q \in \Z^2 \setminus \{0\}$ and $q^\perp \cdot n = 0$, then $R(n) = \{n\}$.
    Otherwise, $R(n) = (n + \operatorname{Span}_\Z(I)) \setminus \{0\}$.
\end{lemma}
\begin{proof}
    If $\operatorname{Span}_\Z(I) = q\Z$ for some $q \in \Z^2 \setminus \{0\}$ and $q^\perp \cdot n = 0$, then for all $k \in I$ we have $k^\perp \cdot n = 0$, so there are no edges involving $n$ and thus $R(n) = \{n\}$.
    
    If $\operatorname{Span}_\Z(I) = q\Z$ for some $q \in \Z^2 \setminus \{0\}$ and $q^\perp \cdot n \neq 0$, then for all $m \in n + \operatorname{Span}_\Z(I)$ and $k \in I \subset q\Z$ we have $k^\perp \cdot m = k^\perp \cdot n \neq 0$, so $(m,m+k) \in E$, thus $R(n) = n + \operatorname{Span}_\Z(I)$.

    The only remaining case is where $\operatorname{Span}_\Z(I)$ cannot be written as $q\Z$, and here it suffices to show that $n + k \in R(n)$ for all $k \in I \setminus \{-n\}$. If $k^\perp \cdot n \neq 0$ then $(n,n+k) \in E$, and thus $n + k \in R(n)$. Otherwise, $k^\perp \cdot n = 0$, and thus $n = rk$ for some $r \in \R \setminus \{0,-1\}$. Let $\ell \in I$ be such that $k^\perp \cdot \ell \neq 0$ (such an $\ell$ exists since otherwise $\operatorname{Span}_\Z(I) \subset k\R$), and note that $\ell^\perp \cdot n = -r(k^\perp \cdot \ell) \neq 0$. Then $(n,n+\ell) \in E$. Also, $(n+\ell,n+\ell+k) \in E$ since $k^\perp \cdot (n + \ell) = k^\perp \cdot \ell \neq 0$. Since $I$ is symmetric we have also $-\ell \in I$, and so $(n + \ell + k, n +k) \in E$ because $(-\ell)^\perp \cdot (n + \ell + k) = -(r+1)(\ell^\perp \cdot k) \neq 0$. This shows $n + k \in R(n)$, as desired.
\end{proof}
Using \Cref{lem:old-thm} and the two lemmas above, we prove \Cref{thm:main}:
\begin{proof}
    If $I$ contains two nonparallel vectors, then by \Cref{lem:rm-characterization} the set $\big\{R(n) \mid n \in \Z^2 \setminus \{0\}\big\}$ is finite. Indeed, $\operatorname{Span}_\Z(I)$ is a finite-index subgroup of $\Z^2$, so each $R(n) = (n + \operatorname{Span}_\Z(I)) \setminus \{0\}$ corresponds to a coset. Thus, we may choose $m_1,\dots,m_N \in \Z^2 \setminus \{0\}$ such that $R(n) \cap \{m_1,\cdots,m_N\}$ is nonempty for all $n \in \Z^2 \setminus \{0\}$. Suppose $s \geq 0$ and $m_j \in \operatorname{supp} \hat \theta_s$. Let $S = S_j$ be as in \eqref{eq:S-def} with $m = m_j$, and note that $m_j \in S$. The Markov property and \Cref{lem:old-thm} applied at the initial condition $\theta_s$ imply, almost surely on the event $\{m_j \in \operatorname{supp} \hat \theta_s\}$, that
    \begin{equation}\label{eq:first-case}
        \lambda_\kappa(\theta_0) \geq -\sup_{n \in S} \Bigg(\kappa |n|^2 + \frac 12 \sum_{k \in I} \sigma_k^2(k^\perp \cdot n)^2\Bigg)\,.
    \end{equation}
    Since $\kappa \in (0,1]$ and $S$ is chosen from finitely many finite sets $S_1,\dots,S_N$, the right-hand side is uniformly bounded below by $-C$, where $C$ depends only on $I$ and $\sigma_k$. By \Cref{lem:acc}, $\cup_{j = 1}^N \cup_{s \in [0,1] \cap \Q} \{m_j \in \operatorname{supp} \hat \theta_s\}$ occurs almost surely, and thus so does \eqref{eq:first-case}, which proves \eqref{eq:first-thm-assertion}.

    Next suppose $\operatorname{Span}_\Z(I) = q\Z$ for some $q \in \Z^2 \setminus \{0\}$ and $n \in \operatorname{supp} \hat \theta_0$ and $q^\perp \cdot n = 0$. By \eqref{eq:dv-eq-1} and $k^\perp \cdot n = 0$ for all $k \in I$, we have that $v_n(t) = e^{-\kappa |n|^2 t} v_n(0)$, and thus
    \begin{equation}\label{eq:kappa-n-squared}
        \lambda_\kappa(\theta_0) \geq -\kappa |n|^2 \quad \text{a.s.}
    \end{equation}

    If $\operatorname{Span}_\Z(I) = q\Z$ for some $q \in \Z^2 \setminus \{0\}$ and $n \in \operatorname{supp} \hat \theta_0$ and instead $q^\perp \cdot n \neq 0$, then by \Cref{lem:rm-characterization} there exists some $m \in R(n)$ such that $k^\perp \cdot m = k^\perp \cdot n$ for all $k \in I$ and $|m|^2 \lesssim (q^\perp \cdot n)^2$. Indeed, we may choose $m = n + rq$ where $r \in \Z$ minimizes $|m|^2$. Let $S$ be as in \eqref{eq:S-def} with this choice of $m$. As above, the Markov property and \Cref{lem:old-thm} imply, almost surely on the event $\{m \in \operatorname{supp} \hat \theta_t\}$, that
    \[\lambda_\kappa(\theta_0) \geq -\sup_{\ell \in S} \Bigg(\kappa |\ell|^2 + \frac 12 \sum_{k \in I} \sigma_k^2(k^\perp \cdot \ell)^2\Bigg)\,.\]
    Since $\kappa \in (0,1]$ and we have, uniformly over $\ell \in S$ and $k \in I$, that $|\ell|^2 \lesssim |m|^2 \lesssim (q^\perp \cdot n)^2$ and $(k^\perp \cdot \ell)^2 = (k^\perp \cdot m)^2 = (k^\perp \cdot n)^2 \lesssim (q^\perp \cdot n)^2$, we conclude that
    \begin{equation}\label{eq:third-case}
        \lambda_\kappa(\theta_0) \geq -C(q^\perp \cdot n)^2\,,
    \end{equation}
    where $C$ depends only on $I$ and $\sigma_k$. By \Cref{lem:acc}, almost surely there is some $t \in [0,1] \cap \Q$ such that $m \in \operatorname{supp} \hat \theta_t$, so \eqref{eq:third-case} holds almost surely. Along with \eqref{eq:kappa-n-squared}, this proves \eqref{eq:second-thm-assertion}.
\end{proof}

\appendix
\crefalias{section}{appendix}
\section{Limit Theorems}\label{sec:appendix}
Below we recall two facts which are essential to the proof of \Cref{thm:main}:

\begin{lemma}\label{lem:martingale}
Let $M_t$ be a continuous real local martingale defined for all $t\ge0$ with $M_0=0$. Then
\begin{equation}\label{eq:martingale-limsup}
 \limsup_{t\to\infty}\frac{M_t}{t}\ge0
 \qquad\text{almost surely}\,.
\end{equation}
\end{lemma}

\begin{proof}
If $\langle M\rangle_\infty<\infty$, then $M_t$ converges almost surely to a finite random variable, and hence $M_t/t\to0$.

Suppose $\langle M\rangle_\infty=\infty$. By the Dambis--Dubins--Schwarz theorem, there is a standard one-dimensional Brownian motion $B$ such that
\begin{equation}\label{eq:DDS}
 M_t=B_{\langle M\rangle_t}.
\end{equation}
Almost surely, there exists a sequence $s_j\to\infty$ such that $B_{s_j} = 0$ (for example, one can take $s_j \coloneqq \inf\{t \geq s_{j-1} + 1 \mid B_t = 0\}$ and $s_0 = 0$). Since $t\mapsto\langle M\rangle_t$ is continuous and increasing, for each $j$ there is $t_j$ with
\[
 \langle M\rangle_{t_j}=s_j.
\]
Since $s_j \to \infty$, we have $t_j\to\infty$, and \eqref{eq:DDS} gives $M_{t_j}=0$. Therefore \eqref{eq:martingale-limsup} holds.
\end{proof}

The following is a consequence of the multiplicative ergodic theorem \cite{met-oldest}:

\begin{lemma}\label{lem:met}
    For every nonzero initial condition $\theta_0 \in L^2_0(\T^2)$, the limit in \eqref{eq:lambda-kappa-def} exists almost surely in $[-\infty,0]$.
\end{lemma}

\begin{proof}
Define the linear map $\Phi(t,\omega): L^2_0(\T^2) \to L^2_0(\T^2)$ by $\Phi(t,\omega) \theta_0 = \theta_t(\omega)$, where $\theta_t$ solves \eqref{eq:white-in-time} with the initial condition $\theta_0$. Here, $\omega \in \Omega \coloneqq C([0,\infty),H_I)$ (recall \eqref{eq:h-i-def}) equipped with the Wiener measure, meaning $W^k_t(\omega) = \omega(t)_k$ for $k \in I$. The shift map $\tau(\omega) = \omega(1+\cdot) - \omega(1)$ is measure-preserving and
\[\Phi(n,\omega) = \Phi(1,\tau^{n-1}\omega)\cdots \Phi(1,\omega)\,.\]
Let $T(\omega)\coloneqq \Phi(1,\omega)$, and note that $\omega\mapsto T(\omega)$ is strongly measurable because $\omega \mapsto \theta_1(\omega)$ is measurable for fixed $\theta_0$, $L_0^2(\T^2)$ is separable, and $\|T(\omega)\| \leq 1$ by \eqref{eq:ln-theta} (indeed, the measurability of $\omega \mapsto T(\omega)\theta$ on a countable dense set of $\theta$'s can be extended using $\|T(\omega)\| \leq 1$ to obtain strong measurability of $\omega \mapsto T(\omega)$). $\|T(\omega)\| \leq 1$ also implies that $\E[\ln^+ \|T(\omega)\|] = 0 < \infty$. By \cite[Proposition B.1]{fourmodes}, $T(\omega)$ is compact for almost every
$\omega$. Therefore, the multiplicative ergodic theorem for compact operators
\cite[Corollary 2.2]{met-old} implies that, for every fixed
$\theta_0\in L^2_0(\mathbb{T}^2)\setminus\{0\}$, almost surely the limit
\[
\lambda(\omega,\theta_0)
:=
\lim_{n\to\infty}
\frac{1}{n}\log\|\Phi(n,\omega)\theta_0\|_{L^2}
\]
exists in $[-\infty,0]$. To pass from integer times to arbitrary times, simply note that $t\longmapsto \|\Phi(t,\omega)\theta_0\|_{L^2}$
is nonincreasing by \eqref{eq:ln-theta}.
\end{proof}

\section*{Acknowledgments}
    The author thanks Keefer Rowan and Wenhao Zhao for insightful discussions on this topic, and Martin Hairer for helpful suggestions to improve the presentation of the paper.
\section*{AI Statement}
    The author asked ChatGPT (OpenAI) to prove a uniform lower bound on $\lambda_\kappa$, and it suggested ideas which now appear in \Cref{lem:S-prop} and \Cref{lem:main-lem}. In particular, it defined $S = \big\{\sum_{k \in I} a_k k \mid a_k \in \{-1,0,1\}\big\} \setminus \{0\}$ and proved a lower bound on $\limsup_{t \to \infty} \frac{1}{t} \ln \|\theta_t\|_{L^2}$ (but not $\lambda_\kappa$ itself) for initial data with $P_S \theta_0 \neq 0$, citing ideas from \cite{fourmodes}. The author subsequently verified and reworked the proofs to be more streamlined, and, also using ideas from \cite{fourmodes}, extended the result to apply to $\lambda_\kappa$ (\Cref{lem:old-thm}). By developing the current family of sets $S$ in \eqref{eq:S-def} and the arguments which appear in \Cref{sec:proof-section}, the author formulated and proved \Cref{thm:main}, which handles arbitrary initial data (removing the assumption $P_S \theta_0 \neq 0$) and shows how the algebraic structure of $I$ affects the behavior of $\lambda_\kappa$.
After further prompting, AI also formulated the example in \Cref{rem:frobenius-failure}, and a later GPT model separately suggested the idea behind \Cref{rem:alternative-proof}.
AI was also used to catch small errors and generate \Cref{fig:rank-one-coset}. The author independently verified all mathematical arguments, wrote the final manuscript, and takes full responsibility for the mathematical content and citations.

\end{document}